\documentclass[12pt]{amsart}
\usepackage{amsmath}
\usepackage{amssymb}
\usepackage{amsthm}
\usepackage{latexcad}
cm \hoffset -1.5 true cm \marginparwidth= 2 true cm

\newtheorem{thm}{Theorem}[section]

\newtheorem{lem}[thm]{Lemma}
\newtheorem{prop}[thm]{Proposition}
\theoremstyle{definition}

\theoremstyle{remark}

\numberwithin{equation}{section}

\newcommand{\supp}{\text{supp }}

\newcommand{\schr}{e^{it\Delta}}

\newcommand {\bbr}{\mathbb{R}}

\newcommand {\la}{\lambda}

\newcommand {\al}{\alpha}

\newcommand {\ep}{\epsilon}

\newcommand {\del}{\delta}
\newcommand {\alphla}{(-\Delta){}^{\frac\alpha2}}
\newcommand {\alphlah}{(-\Delta){}^{\frac\alpha2}}
\newcommand {\propa}{e^{it\alphla}}
\newcommand {\propah}{e^{it'\alphlah}}
\newcommand {\duhamel}{e^{i(t-s)\alphla}}
\newcommand {\mix}[2]{L^{#1}_tL^{#2}_x}
\newcommand {\mi}[1]{L^{#1}_{\xi,\eta_1}}
\newcommand {\betaal}{\beta(\alpha,q,r)}

\newcommand \tx[2] {L^{#1}_tL^{#2}_x}
\newcommand \txp[2] {L^{#1}_{t'}L^{#2}_{x'}}
\newcommand {\intn} {\int_{t_n}^{t_{n+1}}}
\newcommand {\intd} {\int_{\bbr^d}}
\newcommand {\haf} {\widehat{f_k}}
\newcommand {\sn} {\sum_{k=1}^N}
\newcommand {\fl} {\|f\|_{L^2}}

\begin{document}
\title[mass concentration for Nonlinear Schr\"odinger equations]
{Mass concentration for the $L^2$-critical Nonlinear
Schr\"odinger equations of higher orders}
\author{Myeongju Chae}
\author{Sunggeum Hong}
\author{Sanghyuk Lee}

\address{Department of Applied Mathematics\\
        Hankyong National University\\
        Ansong 456-749, Korea}
\email{mchae@hknu.ac.kr}

\

\address{Department of Mathematics\\
        Chosun University\\
        Gwangju 501-759, Korea}
\email{skhong@chosun.ac.kr }

\address{School of Mathematical Sciences, Seoul National University, Seoul, Korea}
\email{shklee@snu.ac.kr}

\thanks{2000\textit{Mathematics Subject Classification. 35B05, 35B30, 35B33, 35Q55, 42B10} }
\thanks{\textit{Key words and phrases. Schr\"odinger equations, high orders, mixed norm blow-up,
mass concentration.}  }
\thanks{}
\begin{abstract}
 We consider the mass concentration phenomenon for the $L^2$-critical
 nonlinear Schr\"odinger equations of higher orders. We show that any solution
 $u$ to $iu_{t} + (-\Delta){}^{\frac\alpha2} u =\pm |u|^\frac{2\alpha}{d}u$, $u(0,\cdot)\in L^2$  for $\al >2$,
 which blows up in a finite time, satisfies a mass concentration phenomenon near the blow-up time.
 We verify that  as $\al$ increases, the size of region capturing a mass concentration
 gets wider due to the stronger dispersive effect.
 \end{abstract}
\maketitle
\section{introduction}\label{intro}
We consider the  $L^2$-critical Cauchy problem in $\bbr^d$, $d\ge 2$,
\begin{eqnarray}\label{alphsch}
\begin{cases}
iu_{t} + (-\Delta){}^{\frac\alpha2} u =\pm |u|^\frac{2\alpha}d u, ~~(t,x)\in \mathbb R_+\times \mathbb R^d  \\
u(0,x) = u_0(x) .
\end{cases}
\end{eqnarray}
Here $(-\Delta){}^{\frac\alpha2}$ is the pseudo-differential operator
defined by
\[(-\Delta){}^{\frac\alpha2} f(x)=(2\pi)^{-d}
\int_{\mathbb R^d} e^{ix\xi} |\xi|^\alpha\widehat f(\xi) d\xi\]
and $\widehat f(\xi)=\int_{\mathbb R^d} e^{-ix\xi}\widehat f(\xi) d\xi$.
The equation \eqref{alphsch} is $L^2$-critical in the sense the that the equation is invariant under the rescaling transformation
$u\to u_\lambda$, $u_\lambda(t,x)=\lambda^\frac d2 u(\lambda^\alpha t,\lambda x)$, which preserves $L^2$ norm.
The system  conserves the mass $M(t)$ and the energy $E(t)$ \textit{a priori};
\begin{align*}
 &M(t) \ = \  \int_{\bbr^d} |u(t,x)|^2 \, dx   ,\\
&E(t) \ = \ \int_{\bbr^d} |(-\Delta)^{\frac {\al}{4}}u(t,x)|^2 \pm (\frac {\al}{d} +1 )^{-1} |u(t,x)|^{\frac {2\al}{d}+2} \, dx.
\end{align*}
The Fourth  order Schr\"odinger equations were initially studied by Karpman \cite{K} and  Karpman  and Shagalov \cite{Ks}.
They considered the fourth order Schr\"odinger equation to take into account the role of small fourth order
dispersion terms in the propagation of intense laser beams in a bulk medium with a cubic nonlinearity (Kerr nonlinearity).
The fourth order $L^2$- critical case with nonlinearity $|u|^{\frac {8}{d}}u$  in  \eqref{alphsch}  was studied in \cite{FIP, p1}.
When $\al\neq 2$, it is unknown whether there exists a blow up solution of \eqref{alphsch} except the numerical evidence
of \cite{FIP}; unlike $\al=2$ case
a virial  type inequality or  a pseudo conformal type  symmetry are not yet known to hold.


In this paper we are concerned with the mass concentration phenomena of  blowup solutions to \eqref{alphsch},
especially when the initial datum $u_0\in L^2$ and its mixed $\tx qr$-norm blows up in a finite time.
When $\alpha=2$ and $d=2$, Bourgain in his seminal paper \cite{bo2} showed that if the
$L^2$-wellposed solution in $\bbr^2$ breaks down at a maximal
time $0 < T^* < \infty$ with
$
\|u\|_{L_{t,x}^{\frac{2(d+2)}d}([0,T^{*}) \times \mathbb{R}^d)}=\infty,
$
then the blow-up solution has a mass concentration phenomenon:
\begin{eqnarray*}
\limsup_{t\nearrow \, T^*} \sup_{
x\in \mathbb{R}^d} \int_{
B(x,(T^*-t)^{\frac12})} |u(t,x)|^2 dx \ge
\epsilon
\end{eqnarray*}
where $\epsilon=C\|u_0\|_2^{-M}$ for some $M>0$. Later, this was
extended to higher dimensions by B\'egout and Vargas \cite{bv}. A generalization in mixed norm spaces $\mix q r$  was obtained
in \cite{ckh}.

\

We consider the case $\alpha > 2$.
The linear part of the equation \eqref{alphsch} has stronger dispersion, compared to the case $\alpha=2$, which may be explained using a following heuristics (see p. $59$ in \cite{tao}).
The plane wave $u(t,x) = e^{ix\cdot\xi_0 + it|\xi_0|^{\al}}$ solves
\begin{align*}
\begin{cases}
  i\partial_t u + (-\Delta)^{\frac {\al}{2}} u=0,\\
   \widehat u_0(\xi)=\del_{\xi_0}.
   \end{cases}
\end{align*}
In order to get a sufficiently broad band solution, still around $\xi_0$, we define
\[u(t,x) = e^{ix\cdot\xi_0 + it|\xi_0|^{\al}}\phi(\ep(x+\al|\xi_0|^{\al-1}\xi_0 t))\]
for a smooth bounded function $\phi$, then $u$ can be shown to satisfy
$ i\partial_t u + (-\Delta)^{\frac {\al}{2}} u = O_{\phi}(\ep^2).$
This means that the profile of $|u|$ moves roughly at velocity $-\al |\xi_0|^{\al-1}\xi_0$. Assuming high frequency initial data
($|\xi_0|\gg1 $), the propagation speed increases as $\alpha$ increases. In other words, the wave tends to spend shorter time in a fixed region.
So, when $\alpha>2$ it is reasonable to expect that we need a larger set for concentration region than $B(x,(T^*-t)^{\frac12})$
to capture nonzero mass in the set. Now, considering the scaling invariance $u\to u_\lambda$  of the equation
\eqref{alphsch}, it is natural to expect that
\begin{eqnarray}
\label{concent}
\limsup_{t\nearrow \, T^*} \sup_{
x\in \mathbb{R}^d} \int_{
B(x,(T^*-t)^{\frac1\alpha})} |u(t,x)|^2 dx \ge
\epsilon>0
\end{eqnarray}
for a  solution $u$ which blows up at $T^*$.
If we impose further condition that  $\epsilon$ depends only  on $\|u\|_0$, then
among the power type sizes, $(T^*-t)^\beta$,  we can see that the size $(T^*-t)^{\frac1\alpha}$ is optimal by a simple
scaling argument.

\

Before giving precise statement of our results, we briefly clarify the issue of  wellposedness of \eqref{alphsch}.
 The local well-posedness in $H^s (\bbr^d), s\ge 0$ relies on the space time estimate for the free propagator
\[ e^{it\alphla}f(x)=\int_{\mathbb R^d} e^{ix\xi+it|\xi|^\alpha} \widehat f(\xi) d\xi,\]
which is called \emph{Strichartz's estimate}(see \eqref{strihomo} in Lemma \ref{stri}).
We call that a pair $(q,r)$ is $\alpha$-admissible if
\[\frac \alpha q+\frac dr =\frac d2, ~~q,\, r \ge 2 \ \text{ and } \ r\neq \infty.\]
Then by the usual argument it is possible to show
the inhomogeneous Strichartz's estimates \eqref{striinho} (Lemma \ref{stri}) for $\alpha$-admissible $(q,r)$ and $(\tilde q,\tilde r)$.
By Duhamel principle the solution
can be written  as
\begin{equation}
\label{duha} u(t,x)=\propa u_0\mp i\int_0^t \duhamel (|u|^\frac{2\alpha}d(s)u(s)) ds. \end{equation}

When  the initial datum $u_0\in L_x^2(\bbr^d)$, following the  standard argument for local wellposedness,  we see that
 there exists the unique solution $u(t,x)$ on a small time interval $[0,T]$ such that
\[u\in C([0,T];L^2(\bbr^d))\cap L^q([0,T];L^r(\bbr^d))\]
whenever $(q,r)$ is $\alpha$-admissible and
\begin{equation}\label{r-range}
 \max\, (\frac d{2(d+2\alpha)}, \frac{d-\alpha}{2d})\le \frac 1r\le \frac{d+\alpha}{2(d+2\alpha)}.
\end{equation}
The existence time interval $[0,T]$ is extended as long as
$
\|u\|_{\mix q r([0,T] \times \mathbb{R}^d)}<\infty.
$
If the solution blows up at $T^*$, then
 \begin{equation}\label{blowup-qr-1}
\|u\|_{L_{t}^{q}L_{x}^{r}([0,T^{*}) \times \mathbb{R}^d)}=\infty.
\end{equation}

Indeed, using \eqref{duha}, the inhomogeneous Strichartz's estimate \eqref{striinho} and H\"older's inequality,  one get for any $\alpha$-admissible $(q,r)$ and
$(\widetilde q, \widetilde r)$
\begin{eqnarray*}
& &\big\|\int_0^T e^{it(-\Delta)^{\frac{\alpha}{2}}(t-s)}
[|u(s)|^{\frac {2\alpha}d}u (s)-|v(s)|^{\frac {2\alpha}d}v (s)] ds\big\|_{\tx qr} \\
&\le& \||u(s)|^{\frac {2\alpha}d}u(s) -|v(s)|^{\frac {2\alpha}d}v(s) \|_{\tx
{\tilde q'}{\tilde r'}}\\
 &\le& C \, (\|u\|_{\tx {q_0}{r_0}}^{\frac {2\alpha}d}+\|u\|_{\tx
{q_0}{r_0}}^{\frac{2\alpha}d})\|u-v \|_{\tx {q_0}{r_0}}
 \end{eqnarray*}
where
$(\frac1{q_0},\frac1{r_0})=\frac{2\alpha+d}d(\frac1{\tilde{q}'},\frac1{\tilde{r}'}).$
  Hence the nonlinear map becomes a contraction map if there are $\alpha$-admissible pairs $(q,r)$ and $(\tilde q, \tilde r)$ satisfying \begin{equation}\label{qr}
(\frac {2\alpha}d +1)\frac1q= \frac1{\tilde{q}'}, \quad(\frac {2\alpha}d +1)\frac1r=\frac1{\tilde{r}'}.
\end{equation}
It is possible as long as the condition \eqref{r-range} is satisfied (see Figure \ref{fig:pre-result}).


\begin{figure}[t]
\label{fig:pre-result}
\centering \setlength{\unitlength}{0.032in}
\begingroup\makeatletter\ifx\SetFigFont\undefined%
\gdef\SetFigFont#1#2#3#4#5{%
  \reset@font\fontsize{#1}{#2pt}%
  \fontfamily{#3}\fontseries{#4}\fontshape{#5}%
  \selectfont}%
\fi\endgroup%
{\renewcommand{\dashlinestretch}{30}
\begin{picture}(104,94)
\thinlines
\drawframebox{60.0}{40.0}{80.0}{80.0}{}
\drawdashline{60.0}{80.0}{60.0}{0.0}
\drawdashline{20.0}{40.0}{100.0}{40.0}
\drawpath{60.0}{0.0}{28.0}{40.0}
\drawpath{60.0}{80.0}{92.0}{40.0}
\drawhollowdot{60.0}{80.0}
\drawhollowdot{92.0}{40.0}
\drawhollowdot{28.0}{40.0}
\drawhollowdot{60.0}{0.0}
\drawcenteredtext{62.0}{83.0}{$c$}
\drawcenteredtext{94.0}{45.0}{$d$}
\drawcenteredtext{32.0}{43.0}{$a$}
\drawcenteredtext{62.0}{3.0}{$b$}
\drawcenteredtext{35.5}{30.6}{$\bullet$}
\drawcenteredtext{47.8}{15.4}{$\bullet$}
\drawpath{28.0}{1.0}{28.0}{-1.0}
\drawpath{92.0}{1.0}{92.0}{-1.0}
\drawcenteredtext{15.0}{40.0}{$\frac12$}
\drawcenteredtext{60.0}{-5.0}{$\frac12$}
\drawcenteredtext{28.0}{-5.0}{$\frac{d-\alpha} {2d}$}
\drawcenteredtext{18.0}{-5.0}{$O$}
\drawcenteredtext{92.0}{-5.0}{$\frac{d+\alpha}{2d}$}
\drawcenteredtext{16.0}{80.0}{$1$}
\drawcenteredtext{100.0}{-5.0}{$1$}
\drawcenteredtext{15.0}{88.0}{$\frac1q$}
\drawcenteredtext{108.0}{-5.0}{$\frac1r$}
\drawcenteredtext{40.0}{32.0}{$A$}
\drawcenteredtext{48.0}{20.0}{$B$}
\drawdotline{20.0}{0.0}{92.0}{40.0}
\drawdotline{60.0}{80.0}{20.0}{0.0}
\end{picture}

}
\vspace{0.2in}

\caption{\small The line segment $[a,b]=[ (\frac{d-\alpha}{2d},\frac12), (\frac12,0)]$ stands
for  $(\frac1r,\frac1q)$  of admissible pair $(q,r)$, and
$[c,d]=[(\frac12,1), (\frac{d+\alpha}{2d},\frac12)]$ stands for $(\frac1{\widetilde r'}, \frac1{\widetilde q'})$
of the dual exponents of admissible pairs $(\tilde q,\tilde r)$.
For $(\frac1r,\frac1q)$ in the segment $[A,B]$ we can find admissible pairs $(\tilde q, \tilde r)$ satisfying the relation \eqref{qr}. }
\end{figure}


The following is our first result.
\begin{thm}\label{elliptic}
 Let $(q,r)$ be an $\alpha$-admissible pair satisfying  $q > 2$ and \eqref{r-range}.
  Suppose that the solution of \eqref{alphsch} satisfies $\|u\|_{L_{t}^{q}L_{x}^{r}([0,t) \times \mathbb{R}^d)}
<\infty$ for $0 < t < T^{*}<\infty$ and
\eqref{blowup-qr-1}. Then \eqref{concent} holds
\end{thm}

The results in \cite{bv, bo2, ckh}  were obtained by the use of  refinement of Strichartz's estimates for $\schr f$
which come from bilinear restriction estimate for the paraboloid \cite{mvv,lv,ta,tvv}.
To deal with the case $\alpha>2$ we need similar estimates for $\propa$.
It turns out that the related analysis is simpler than \cite{bv, bo2, ckh}
due to a stronger dispersion effect so that we  give a direct proof of  refinement of Strichartz's estimates
for $\propa$ exploiting bilinear interaction of Schr\"odinger waves. In particular we have the refinement
 (Proposition \ref{lem-con}) in terms of dyadic shells, instead of cubes as in the previous work \cite{bv, bo2, ckh}
 for which  the Galilean invariance of the operator $\schr f$ played a role, which is no longer available
 when $\alpha\neq 2$.

\

\begin{figure}[t]
\label{fig:pre-result2}
\centering \setlength{\unitlength}{0.032in}
\begingroup\makeatletter\ifx\SetFigFont\undefined%
\gdef\SetFigFont#1#2#3#4#5{%
  \reset@font\fontsize{#1}{#2pt}%
  \fontfamily{#3}\fontseries{#4}\fontshape{#5}%
  \selectfont}%
\fi\endgroup%
{\renewcommand{\dashlinestretch}{30}

\begin{picture}(120,94)
\thinlines
\drawframebox{60.0}{40.0}{80.0}{80.0}{}
\drawdashline{60.0}{80.0}{60.0}{0.0}
\drawdashline{20.0}{40.0}{100.0}{40.0}
\drawpath{34.0}{40.0}{60.0}{0.0}
\drawdashline{60.0}{80.0}{86.0}{40.0}
\drawhollowdot{60.0}{80.0}
\drawhollowdot{110.0}{40.0}
\drawhollowdot{82.0}{80.0}
\drawhollowdot{86.0}{40.0}
\drawhollowdot{34.0}{40.0}
\drawhollowdot{60.0}{0.0}
\drawcenteredtext{62.0}{83.0}{$c$}
\drawcenteredtext{112.0}{43.0}{$f$}
\drawcenteredtext{82.0}{83.0}{$e$}
\drawcenteredtext{88.0}{43.0}{$d$}
\drawcenteredtext{36.0}{43.0}{$a$}
\drawcenteredtext{62.0}{3.0}{$b$}
\drawdotline{82.0}{80.0}{20.0}{0.0}
\drawthickdot{41.8}{28.0}
\drawthickdot{51.0}{13.8}
\drawcenteredtext{16.0}{40.0}{$\frac12$}
\drawcenteredtext{16.0}{80.0}{$1$}
\drawpath{110.0}{1.0}{110.0}{-1.0}
\drawpath{86.0}{1.0}{86.0}{-1.0}
\drawpath{34.0}{1.0}{34.0}{-1.0}
\drawpath{82.0}{80.0}{110.0}{40.0}
\drawpath{100.0}{0.0}{118.0}{0.0}
\drawdotline{20.0}{0.0}{110.0}{40.0}
\drawcenteredtext{60.0}{-5.0}{$\frac12$}
\drawcenteredtext{16.0}{-5.0}{$O$}
\drawcenteredtext{100.0}{-5.0}{$1$}
\drawcenteredtext{34.0}{-5.0}{$\frac{d-\alpha}{2d}$}
\drawcenteredtext{120.0}{-5.0}{$\frac1r$}
\drawcenteredtext{110.0}{-5.0}{$\frac{3d-2\alpha}{2d}$}
\drawcenteredtext{86.0}{-5.0}{$\frac{d+\alpha}{2d}$}
\drawcenteredtext{46.0}{28.0}{$A$}
\drawcenteredtext{52.0}{18.0}{$B$}
\drawcenteredtext{16.0}{87.0}{$\frac1q$}
\end{picture}

}
\vspace{0.2in}

\caption{\small The line segments $[a,b]$  and
$[c,d]$ are the same as in Figure \ref{fig:pre-result}.
The line segment $[e,f]$ corresponds to $(\frac1{\widetilde q'},\frac1{\widetilde r'})+(0,\frac{d-\alpha}d)$.
For $(\frac1r,\frac1q)$ in the segment $[A,B]$ we can find admissible pairs $(\tilde q, \tilde r)$ satisfying the relation \eqref{hart-qr}. }
\end{figure}
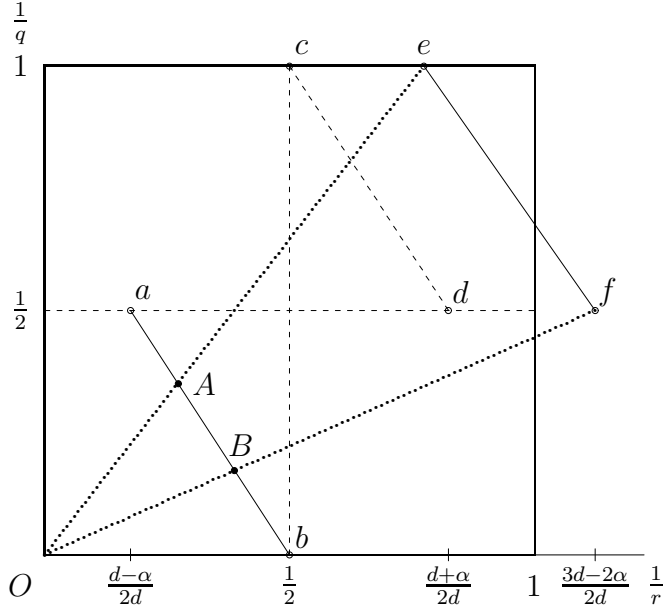


Secondly, we consider the  $L^2$-critical Hartree equation, which
is given by for $2<\alpha <d$
\begin{eqnarray}\label{hartee}
\begin{cases}
iu_{t} + (-\Delta )^\frac\alpha 2 u=  \pm( |x|^{-\alpha}\ast |u|^2)u \\
u(0,x) = u_0(x) \in L^{2}(\mathbb{R}^d), \quad d\ge 3.
\end{cases}
\end{eqnarray}
One can easily check that the equation \eqref{hartee} is also $L^2$-critical, that is, invariant under $u\to u_{\la}$.
  One may be interested in a mass concentration for
the finite time blow-up solutions for \eqref{hartee}. The local wellposedness can be established by following the standard
argument. In fact, using the Strichartz estimates \eqref{striinho} and triangle inequality
\begin{eqnarray*}
& &\big\|\int_0^T e^{it(-\Delta)^{\frac{\alpha}{2}}(t-s)}
[(|x|^{-\alpha}\ast |u|^2)u-(|x|^{-\alpha}\ast |v|^2)v] ds\big\|_{\tx qr}
\\
&\le& \|(|x|^{-\alpha}\ast [|u|^2-|v|^2])u\|_{\tx
{\tilde q'}{\tilde r'}}+\|(|x|^{-\alpha}\ast |v|^2)(v-u) \|_{\tx
{\tilde q'}{\tilde r'}}.
 \end{eqnarray*}
By H\"older's inequality and Hardy-Littlwood-Sobolev inequality the last of the above is bounded by
\begin{equation}\label{ok}
C\||u|^2-|v|^2\|_{\tx {q_1}{r_1}}\|u\|_{\tx qr}+C\|v\|_{\tx {2q_1}{2r_1}}^2\|u-v\|_{\tx qr}
\end{equation}
for $q_1,r_1$ satisfying
$(\frac 1{q_1}, \frac 1{r_1})+(\frac1q,\frac1r)=(\frac1{\widetilde q'},\frac1{\widetilde r'})+(0,\frac{d-\alpha}d).$
Let us take $q_1= \frac q 2$ and $r_1= \frac r 2$. Then we find that
the nonlinear map
\[u\to \propa u_0\mp i \int_0^t \duhamel (|x|^{-\alpha}\ast |u|^2)u \, ds\]
is a contraction if there is an $\alpha$-admissible pair $(\widetilde q, \widetilde r)$ such that
\begin{equation}\label{hart-qr}(\frac3q,\frac3r)=(\frac1{\widetilde q'},\frac1{\widetilde r'})+(0,\frac{d-\alpha}d).
\end{equation}
An easy calculation shows that  the line segment $[A,B]$ in Figure $2$ is parallel to the segment $[e,f]$
corresponding to the set $\{(1/{\widetilde q'},1/{\widetilde r'})+(0,(d-\alpha)/d):
 ( \widetilde{q}, \widetilde{r}) \text{ is } \al-admissible\}$, and moreover $|e-f|=3|A-B|$.
So it is possible to find $(\widetilde q, \widetilde r)$ satisfying \eqref{hart-qr} as long as
 $(q,r)$ is contained in $[A,B]$, that is,
\begin{equation}\label{r-gange-hart}
\frac{6d}{3d-\alpha}\le r\le \frac{6d}{3d-2\alpha}.
\end{equation}
For these $(q,r)$ we also get a blowup alternative; If  $T^*<\infty$, then \eqref{blowup-qr-1} should be satisfied.
As it was shown in \cite{ckh}, the mass concentration phenomenon is mostly involved with
the homogeneous part of the solution. The argument used in
\cite{bv, bo2} works for \eqref{hartee}
without much modifications if  the
nonlinear term can be controlled properly. This is actually  equivalent to showing the local wellposedness of
\eqref{hartee} under the condition \eqref{r-gange-hart}.

\begin{thm}\label{hartree} Let $d \geq 3$.
Let $(q,r)$ be an $\alpha$-admissible satisfying \eqref{r-gange-hart}.  Suppose that the solution $u$ of \eqref{hartee}
satisfies $\|u\|_{L_{t}^{q}L_{x}^{r}([0,t) \times \mathbb{R}^d)} <\infty$
for $0 < t < T^{*}<\infty$ and \eqref{blowup-qr-1}. Then \eqref{concent} holds. \end{thm}

The paper is organized as follows. In Section 2 we obtain
some preliminary  estimates which are to be used for the proofs of
Theorems. In Section 3 we give the proofs of Theorems \ref{elliptic} and
\ref{hartree}.


\section{preliminary}
In this section we show several lemmas which will be used later
for the proofs of the theorems.
For $q,r\ge 2$, $r\neq \infty$ and $\frac2q+\frac dr\le \frac d2$, set
\[\beta=\beta(\alpha,q,r)=\frac d2-\frac dr-\frac\alpha q.\]
Let $\rho$ be a smooth function
supported in $[1/2,4]$ and satisfying $\sum_{-\infty}^\infty\rho(x/2^k)=1$ for all $x>0$.
Then we define a projection operator by
\[ \widehat {P_k f}(\xi)=\rho(|\xi|/2^k) \widehat f(\xi).\]

The following lemma is a version of Strichartz estimates for Schr\"odinger equations of higher orders $\alpha$ with $\alpha > 2$.
It seems well known  but for a convenience of the readers we include the proof.
The arguments are based on rescaling  and  Littlewood-Paley theorem.

\begin{lem}\label{stri} For $q,r\ge 2$, $r\neq \infty$ and $\frac 2 q+\frac dr\le \frac d2$,
\begin{equation}\label{vect}
\|\propa f\|_{L^q_tL^r_x}\le C \, (\sum_k 2^{2k(\frac d2-\frac dr-\frac\alpha q)}\| P_k f\|_2^2)^\frac12.
\end{equation}
In particular, if $(q,r)$ is $\alpha$-admissible, then
\begin{equation}\label{strihomo}
\|\propa f\|_{L^q_tL^r_x}\le C \| f\|_2.
\end{equation}
Also if $(q,r)$ and $(\tilde q, \tilde r)$ are $\alpha$-admissible, then we have
\begin{equation}\label{striinho}
\|\int_0^t\duhamel F(s) ds\|_{L^q_tL^r_x}\le C  \, \| F\|_{\mix {\tilde q'}{\tilde r'}}.
\end{equation}
\end{lem}

\begin{proof} Once we get \eqref{vect}, then \eqref{strihomo} follows from
Plancherel's theorem. Also \eqref{striinho} can be shown by duality and the argument due to Christ and Kiselev (\cite{ck}).

We now show \eqref{vect}. Since $\alpha>2$, by the stationary phase method (see p.344 in \cite{st}),
we see $\|\propa \psi \|_{L^\infty_x}\le C \, t^{-\frac d2}$ for any $\psi$ with compact support contained in
 $\mathbb R^d\setminus \{0\}$. Hence, from the argument of Keel-Tao in \cite{kt},
 we have \begin{equation}\label{p0}
 \|\propa P_0 f\|_{\mix{q}{r}}\le C\|f\|_2
 \end{equation}
 whenever $\frac dr+\frac 2 q\le \frac d2$, and $r,q\ge 2$ (with exception $r\neq \infty$ when $d=2$).   Then by rescaling  we observe that
 \begin{equation*}
 \propa P_{k} f(x)= e^{i2^{\alpha k} t\alphla}P_0[f(\frac\cdot{2^k})](2^k x).
 \end{equation*}
 Therefore it follows that
 \begin{equation}\label{rescale}
 \|\propa P_{k} f\|_{\mix q r} \le C2^{(-\frac dr-\frac \alpha q+\frac d2)k}\|f\|_2.
 \end{equation}
 Since $f=\sum_{k} P_kf$ and  $q,r\ge 2$, from Littlewood-Paley theorem followed by Minkowski's inequality
 we have
 \begin{align*}
 \|\propa f\|_{\mix q r}  &\le C(\sum_k\|\propa P_kf\|_{\mix q r}^2)^\frac12.
 \end{align*}
Putting \eqref{rescale} in the right hand side of the above, we get the desired.
\end{proof}

\subsection{Refinement of Strichartz's estimates}

\begin{lem}\label{inter} Let $(q,r)$ satisfy $q>2,r\ge 2$, $r\neq \infty$ and $\frac2q+\frac dr< \frac d2$. If $M\le N$
 then there is $\epsilon=\epsilon(q,r)>0$ such that
 \begin{equation*} \label{better}\|\propa P_Nf\propa P_Mf \|_{\mix {q/2}{r/2}}\le C2^{(M+N)\betaal}\left(2^{M-N}\right)^{\epsilon} \|f\|_2\|g\|_2.\end{equation*}
\end{lem}

This means that  it is possible to obtain better bounds than the one trivially obtained by
rescaling (Lemma \ref{stri}) when  the waves interact at different frequency levels.
Such observation was first made by Bourgain \cite{bo2}.

\begin{proof}
By rescaling it is enough to show that
\begin{equation}
\label{better}\|\propa P_0f\propa P_{M-N}f \|_{\mix {q/2}{r/2}}\le C\left(2^{M-N}\right)^{\betaal+\epsilon} \|f\|_2\|g\|_2.\end{equation}

Let us set $L=M-N\le 0$. Hence Fourier supports of $P_0 f$, $P_L f$ are contained in
the sets $\{|\xi|\sim 1\}$, $\{|\xi|\sim 2^{L}\}$, respectively.
For $\frac dr+\frac2q\le \frac d2$, and $r,q\ge 2$, by H\"older's inequality and \eqref{rescale}
one can see
\[\|\propa P_0f\propa P_Lg\|_{\mix {q/2}{r/2}}\le C2^{L\betaal}\|f\|_2\|g\|_2.\]
If one interpolates this with
\begin{equation}
\label{l2}
\|\propa P_0f\propa P_Lg\|_{\mix 2 2}\le C2^{L(d-1)/2}\|f\|_2\|g\|_2
\end{equation}
which will be proven later,
one get the desired estimat \eqref{better}. Indeed, note that the bound
in the above is better than the trivial bounds follows from rescaling.
That is,
\[2^{L\beta(\alpha,4,4)}=2^{L(d/4-\alpha/4)}>  2^{L(d-1)/2}=2^{L(\beta(\alpha,4,4)+\epsilon)}\]
for some $\epsilon>0$ because $\alpha, d\ge 2$.
Hence via interpolation  we get the desired estimate
\[\|\propa P_0f\propa P_Lg\|_{\mix {q/2}{r/2}}\le C2^{L(\betaal+\epsilon)}\|f\|_2\|g\|_2\]
with some $\epsilon>0$ as long as $d/r+2/q < d/2$ and $q>2$.
\end{proof}

\begin{proof}[Proof of \eqref{l2}] We may assume that  $\widehat f$ is supported in the set $\{\xi:|\xi|\sim 1\}$. When $2^L\sim 1$,
 the estimate \eqref{l2} is trivial from \eqref{rescale} and H\"older's inequality.  So we also may assume
$2^L\ll 1$.

By decomposing  the  Fourier support of $f$ into finite number of
sets, rotation and  mild dilation,
it is enough to show that
\[\|\propa f\propa g\|_{\mix {2}{2}}\le C\lambda^{(d-1)/2}\|f\|_2\|g\|_2\]
whenever $\widehat f$  is supported in $B(e_1, \epsilon)$  and $\widehat g$ is supported in $\{|\xi|\sim 2^{L}\}$. Here $B(x,r)$ is the open ball centered at $x$ with
radius $r$.
We write
\[ \propa f(x)\propa g(x)=\iint e^{i(x(\xi+\eta)+t(|\xi|^\alpha+|\eta|^\alpha))} \widehat f (\xi)\widehat g(\eta) d\xi d\eta.\]
Freezing $\bar\eta=(\eta_2,\dots,\eta_d)$, we  consider an operator
\[B_{\bar\eta}(f,g)=\iint e^{i(x(\xi+\eta)+t(|\xi|^\alpha+|\eta|^\alpha))} \widehat f (\xi)\widehat g(\eta_1,\bar\eta) d\xi d\eta_1.\]
We now make the change of variables
\[\zeta=(\zeta_1, \zeta_2,\dots, \zeta_{d+1})=(\xi+\eta,|\xi|^\alpha+|\eta|^\alpha).\]
Then by a direct computation  one can see that
\[\left|\frac{\partial \zeta}
{\partial(\xi,\eta_1)}\right|=\alpha\big|\eta_1|\eta|^{\alpha-2}-\xi_1|\xi|^{\alpha-2}\big|\sim 1\]
on the supports of $\widehat f$ and $\widehat g$. Hence making change of variables $(\xi,\eta_1)\to \zeta$, applying Plancherel's theorem and
reversing the change variables  ($\zeta \to (\xi,\eta_1)$), we get
\[\|B_{\bar\eta}(f,g)\|_{\mix 2 2}\le C\|\widehat f(\xi) \widehat g(\eta_1,\bar\eta)\|_{\mi 2}. \]
Since
\[\propa f(x)\propa g(x)=\int  B_{\bar\eta}(f,g(\cdot, \bar\eta)) d\bar\eta,\]
by Minkowski's inequality we see
\[\|\propa f\propa g\|_{\mix 2 2} \le C \int  \|\widehat f(\xi) \widehat g(\eta_1,\bar\eta)\|_{\mi 2} d\bar\eta. \]
This gives the desired bound by Schwartz's inequality, because of $|\bar \eta|\le 2^{L}$.
\end{proof}

\begin{prop}\label{lem-con} If $(q,r)$ is an $\alpha$-admissible with $q>2$, there are $\theta\in (0,1)$ and $1\le p<2$ such that
\[\|\propa f\|_{\mix q r}\le C \left(\sup_{k}
2^{kd(\frac 12-\frac1p)}\|{ \widehat f\chi_{B_k}}\|_p\right)^\theta \|f\|_2^{1-\theta}.\]
Here $B_k=\{\xi:2^{k-1}<|\xi|\le 2^k\}$.
\end{prop}

\begin{proof}[Proof of Proposition \ref{lem-con}]
In fact, for the proof it is sufficient to show
that \[\|\propa f\|_{\mix q r}\le C \left(\sup_{k}
2^{kd(\frac 12-\frac1p)}\|\widehat{P_k f}\|_p\right)^\theta \|f\|_2^{1-\theta}.\]
By dividing the support of $\widehat{P_k f}$ into three dyadic shells $B_{k-1}, B_k$, and $B_{k+1}$,
we get the desired. This actually can be shown by using \eqref{vect} and the following two estimates:

If  $(q,r)$ is an $\alpha$-admissible with $ q>2$, then
\begin{align}\label{l4}
\|\propa f\|_{\mix { q} {r}}&\le C(\sum_k \|\widehat {P_k f}\|_2^{ \tilde q})^\frac1{ \tilde q},
\end{align}
and
\begin{align}
\label{lp}\|\propa f\|_{\mix {q} {r}}&\le C(\sum_k (2^{k(\frac d2-\frac d{\tilde p})}\|\widehat {P_k f}\|_{\tilde p})^2)^\frac12
\end{align}
with some $\widetilde p<2<\widetilde q$. Interpolation among \eqref{vect} and these two estimates
gives
\begin{equation}\label{mixed} \|\propa f\|_{\mix q r}
\le C(\sum_k (2^{k(\frac d2-\frac d{p_*})}\|\widehat {P_k f}\|_{p_*})^{q_*})^\frac1{q_*}\end{equation}
as long as  $(1/p_*, 1/q_*)$ is contained in the triangle $\Gamma$
with vertices $(1/2,1/2), (1/2, 1/\tilde q)$ and $(1/\tilde p,1/2)$.
Obviously one can find a point $(1/p_0,1/q_0)$ contained in the interior of $\Gamma$ so that
it lies  on the line segment joining  $(1/2,1/2)$ and $(1/ p,0)$  for some  $p<2$. Then
by interpolation among the mixed norm spaces\footnote{Here the mixed norm spaces are given with the norm
$(\sum_{k} (2^{k(\frac d2-\frac ds)}\|f_k\|_s)^t)^{\frac1t}$.}(\cite{bl})
we see
\[ (\sum_k (2^{k(\frac d2-\frac d{p_0})}\|\widehat {P_k f}\|_{p_0})^{q_0})^\frac1{q_0}\le C\left(\sup_{k}
2^{kd(\frac 12-\frac1p)}\|\widehat{P_k f}\|_p\right)^\theta \left(\sum_k (\|\widehat {P_k f}\|_{2}^{2})^\frac1{2}\right)^{1-\theta}.\]
Therefore, using \eqref{mixed} which is valid with $(q_*,p_*)=(q_0,p_0)$ together with the above and Plancherel's theorem we get the desired inequality.
Now it remains to show \eqref{l4} and \eqref{lp}.

We first show \eqref{lp} which is easier. Note that $\alpha>2$. By interpolation between \eqref{p0} and the trivial $L^1\to L^\infty$ bound, one can see that for each $\alpha$-admissible $(q,r)$, $q>2$, there is a $p<2$ such that
  \[\|\propa P_0 f\|_{\mix{q}{r}}\le C\|\widehat {P_0 f}\|_p. \]
 Here we used the fact that $\alpha>2$. Then by rescaling we see that
\[\|\propa P_k f\|_{\mix{q}{r}}\le C2^{k(d-\frac dp-\frac\alpha q-\frac dr)}\|\widehat {P_k f}\|_p. \]
By using  Littlewood-Paley theorem,  Minkowski's inequality and the above  we
get
\begin{align*}
\|\propa  f\|_{\mix{q}{r}}&\le C(\sum_k \|\propa P_k f\|_{\mix{q}{r}}^2)^\frac12\\
&\le C(\sum_k
2^{2k(d-\frac dp-\frac\alpha q-\frac dr)}\|\widehat {P_k f}\|_p^2)^\frac12.
\end{align*}
In particular, when $(q,r)$ is $\alpha$-admissible we get \eqref{lp}.

\

Now we turn to \eqref{l4}. We start with
the inequality \eqref{vect} which reads as
\[\|\propa f\|_{\mix { q} {r}}\le C(\sum_k 2^{2k\betaal}\|\widehat {P_k f}\|_2^{2})^\frac1{2}\]
for  $q,r\ge 2$, $r\neq \infty$ and $\frac2q+\frac dr\le \frac d2$.
However in the right hand side the norm in $k$ is $\ell^2$.
We need to upgrade this slightly so that the norm in $k$ is replaced by $\ell^{\tilde q}$ for some $\tilde q>2$.
To do this it is enough to show that there is a pair $(q,r)$ satisfying $q,r\ge 2$, $r\neq \infty$
and $\frac2q+\frac dr\le \frac d2$, such that
\begin{equation}
\label{raised}
\|\propa f\|_{\mix { q} {r}}\le C(\sum_k (2^{k\betaal}\|\widehat {P_k f}\|_2)^{\tilde q})^\frac1{\tilde q}
\end{equation}
for some $\tilde q>2$. The interpolation between this and \eqref{vect} gives the desired. In particular
when $(q,r)$ is $\alpha$-admissible we get \eqref{l4}.

We show \eqref{raised} with $q=r=4$ and $\tilde q=4$.
We write
\[\|\propa f\|_{\mix {4}{4}}^2=\|\propa f\propa f\|_{\mix {2} {2}}^2\]
and
\[\propa f\propa f=\sum_{k\le l}\propa P_k f\propa P_l f+\sum_{k> l}\propa P_k f\propa P_l f.\]
Then by triangle inequality
\begin{align*}
\|\propa f\propa f\|_{\mix {2} {2}}\le & \sum_{j\ge 0}\|\sum_{k}\propa P_k f\propa P_{k+j} f\|_{\mix {2} {2}}\\
+\sum_{j>0}&\|\sum_{k}\propa P_{k+j} f\propa P_k f\|_{\mix {2} {2}}.
\end{align*}
By symmetry it is enough to deal with the first one because the second can be handled similarly. Hence it is enough to show that
\begin{equation}
\label{j}
\|\sum_{k}\propa P_k f\propa P_{k+j} f\|_{\mix {2} {2}}\le C2^{-\epsilon j}
(\sum_k 2^{4k\beta(\al,4,4)} \|\widehat{P_k f}\|_2^{4})^\frac2{ 4}
\end{equation}
for some $\epsilon>0$.
We consider separately the cases $j=0,1,2$ and $j\ge 3$.

First we handle the case $j=0,1,2$. By Cauchy-Schwarz's inequality we have
\begin{align*}
|\sum_{k}\propa P_k f\propa P_{k+j} f|& \le \sum_{k}|\propa P_k f|^2.
\end{align*}
So, squaring both sides we get
\begin{align*}
|\sum_{k}\propa P_k f\propa P_{k+j} f|^2&\le
C\sum_{l\ge 0}\sum_{k}|\propa P_k f\propa P_{k+l} f|^2.
\end{align*}
Hence it follows  that
\begin{align*}
\|\sum_{k}\propa P_k f\propa P_{k+j} f\|_{\mix {2} {2}}^2
& \le C\sum_{l\ge 0}\sum_{k}\|\propa P_k f\propa P_{k+l} f\|_{\mix {2} {2}}^2.
\end{align*}
Then by Lemma \ref{inter} we see
\begin{eqnarray*}
\sum_{l\ge 0}\sum_{k}\|\propa P_k f\propa P_{k+l} f\|_{\mix {2} {2}}^2\\
\le C\sum_{l\ge 0}2^{-\epsilon l}\sum_{k}&2^{2k\beta(\alpha,4,4)}\| P_k f\|_2^2\ 2^{2(k+l)\beta(\alpha,4,4)}\|P_{k+l} f\|_2^2.\end{eqnarray*}
Therefore by Schwarz's inequality and summation in $l$ we get
\[\|\sum_{k}\propa P_k f\propa P_{k+j} f\|_{\mix {2} {2}}^2\le C(\sum_{k}(2^{k\beta(\alpha,4,4)}\| P_k f\|_2)^4).\]

Now we turn to case $j\ge 3$. Observe the Fourier supports of
\[\propa P_k f\propa P_{k+j}, \quad -\infty<k<\infty\] are boundedly overlapping. Hence by Plancherel's theorem,
we see that
\[\|\sum_{k}\propa P_k f\propa P_{k+j} f\|_{\mix {2} {2}}^2
\le C \sum_{k}\|\propa P_k f\propa P_{k+j} f\|_{\mix {2} {2}}^2.
\]
Using Lemma \ref{inter}, the right hand side is bounded by
\[C \, 2^{-2\epsilon j}\sum_{k}2^{2k\beta(\alpha,4,4)}\| P_k f\|_2^2 \  2^{2(k+j)\beta(\alpha,4,4)}\|P_{k+j} f\|_2^2.\]
Therefore, Schwarz's inequality gives us the desired bound \eqref{j}.
\end{proof}

Proposition \ref{lem-con} can be combined with the following elementary lemma to find out
the region where the given $L^2$ function is not severely concentrating but still containing a moderate amount of mass.

\begin{lem}\label{nonconcent} Let $\ep>0$, $f\in L^2(\bbr^2)$ and suppose that there is a measurable subset $Q$ such that
\[ \ep\le (|Q|^{\frac12-\frac1p}\|f\chi_Q\|_p)^\theta\|f\|_2^{1-\theta}\]
for some $\theta\in (0,1)$ and $p\in [1,2)$. Then
if
$\lambda\sim |Q|^{-\frac12} \ep^{-\frac p{\theta(2-p)}}\|f\|_2^{\frac p{\theta(2-p)}+1}$,
then $f_Q^\lambda=f\chi_{\{x\in Q: |f|\le \lambda\}}$ satisfies
\[\ep^\frac1\theta\|f\|_2^{1-\frac1\theta}\lesssim |Q|^{\frac12-\frac1p}\|f_Q^\lambda\|_p\le \|f_Q^\lambda\|_2.\]
Here all the implicit constants are independent of $f$, $Q$, $\ep$ and $\lambda$.
\end{lem}

\begin{proof} Changing $|Q|^\frac12 f(|Q|^\frac1d \cdot)/\|f\|_2\to f$, $|Q|^\frac12\lambda/\|f\|_2\to \lambda$,
 and $\ep/\|f\|_2\to \ep $,
   we may assume $|Q|=1$ and $\|f\|_2=1$.
Since
\[ \ep^\frac p\theta \le \int_Q |f|^p dx \le \int_{\{x\in Q: |f|\le \lambda\}} |f|^p dx+\int_{\{x\in Q: |f|> \lambda\}}
 |f|^p dx,\]
it trivially follows that
\[ \ep^\frac p\theta \le \int_{\{x\in Q: |f|\le \lambda\}} |f|^p +\lambda^{p-2}\]
because $\|f\|_2=1$. Now we only need to choose $\lambda$ such that
$\lambda^{p-2}=\frac12 \ep^{\frac p\theta}$. The remaining is easy to see by making the changes of
$f\to |Q|^\frac12 f(|Q|^\frac1d \cdot)/\|f\|_2$, $\lambda\to |Q|^\frac12\lambda/\|f\|_2$ and $\ep/\|f\|_2\to \ep $.
\end{proof}


\section{Proof of Theorems}


As in $\al=2$ case (\cite{bv,bo2,ckh}) the following two lemmas play  crucial roles in showing the mass concentration.
The first one is concerned with decomposition of
the initial datum into functions of which Fourier transforms are
spreading rather than concentrating. In view of uncertainty principle the spreading part of the
initial datum may concentrate on some spatial region.  The second one enables us to find
regions where the linear Schr\"odinger wave concentrates in the mixed norm space $L^q_tL^r_x$
(here $(q,r)$ is $\al$ admissible) when the Fourier
transform of the initial data does not severely concentrate.

\begin{lem}\label{square} Let $(q,r)$ be an  $\alpha$-admissible pair satisfying $q>2$ and $\alpha/q+d/r=d/2$.
Suppose $f \in L^2(\bbr^d)$ and
\begin{equation}\label{lem-assumption}  \big\|
e^{it(-\triangle)^{\frac{\alpha}{2}}}f\big\|_{L^{q}_tL^{r}_x}\,\ge\, \epsilon\end{equation}
for some $\epsilon>0$. Then there
exist a $f_k\in L^2(\mathbb{R}^d)$ and a dyadic shell $B_{n_k}$  for $k=1,2,\cdots,N
$ with $N=N(\|f\|_{L^2},d,\epsilon)$  such that
\begin{itemize}
 \item[(1)]
$ \text{supp}\,\widehat{f_k}\subset
B_{n_k}$, $|\haf| < C2^{-\frac{n_kd}{2}} \ep^{-\nu}\fl^{\mu}$ for all $k=1,\dots, N$,
\item[(2)] $\big\|
e^{it(-\triangle)^{\frac{\alpha}{2}}}f-\sum_{k=1}^Ne^{it(-\triangle)^{\frac{\alpha}{2}}}f_k\big\|_{L^{q}_tL^{r}_x}<\epsilon,$
\item[(3)] $\|f\|^2_{L^2}=
\sum_{k=1}^N\|f_k\|_{L^2}^2+\|f-\sum_{k=1}^Nf_k\|^2_{L^2}.$
\end{itemize}
Here the constants $C$, $\mu$, and $\nu$ depend only on $d$.
\end{lem}
\begin{lem}\label{tube} Let $(q,r)$ be an $\alpha$-admissible pair satisfying $2< q \le r$.
Suppose $g \in L^2(\mathbb R^d)$ and
\[\supp \widehat{g} \subset B_k \ \text{ and } \ |\widehat{g}| < C_0 2^{-\frac {kd}{2}} \]
for  $C_0>0$. Then for any $\ep>0$, there exist $N_1\in \mathbb{N}$,  $N_1\le C(d,C_0,\ep)$,
and sets $(\mathcal{Q}_{n})_{1\le n \le N_1} \subset \bbr \times \mathbb R^d$ which is given by
\begin{align}\label{qn}
 \mathcal{Q}_{n} = \{ (t,x) \in \bbr\times \mathbb R^d ; t\in I_n  \mbox{ and }x\in C_n\},
 \end{align}
  where $I_n \subset \bbr$ is an interval with $|I_n|= {2^{-k\alpha}}$ and $C_n$ is a cube  with the side length  $\it{l}(C_n)=2^{-k}$ such that
  \[\| e^{it(-\triangle)^{\frac{\alpha}{2}}} g\|_{\tx{q}{r}(\bbr^{d+1}\backslash \bigcup _{n=1}^{N_1} \mathcal{Q}_{n})} <\ep.\]
  \end{lem}

{\bf Notation.}
  Let $E$ be a measurable set in $\mathbb R^{d+1}$ and $f:\bbr^{d+1} \to \bbr$ is a measurable function.
  If $E_t=\{ x:(t,x)\in E\}$ is measurable in $\mathbb R^d$ for all $t\in \bbr$, we define the mixed integral $\|f\|^q_{\tx{q}{r}(E)}$ by   \[ \|f\|^q_{\tx{q}{r}(E)} = \int _{\bbr} \big( \int _{E_t} |f(t,x)|^r dx \big)^{\frac qr}
  dt.\]

Once we have the refinement of Strichartz estimates (Proposition \ref{lem-con})
the proofs of Lemma \ref{square}  and \ref{tube} can be given by a modification of the argument in \cite{ bv,bo2}.
The proofs of lemmas are given in Appendix.

\smallskip

{\it Proof of Theorem \ref{elliptic}.}\\
The proof consists of following steps:\\
$\cdot$ \textit {Controlling the inhomogeneous part,  }\\
$\cdot$ \textit {Decomposition to the initial datum with non-concentration Fourier transforms,}\\
$\cdot$ \textit {Figuring out the concentrating region, }\\
$\cdot$ \textit {Determining the size of mass concentration region.}\\
The two lemmas (Lemmas \ref{square} and \ref{tube}) will be incorporated into
the second and the third step respectively.\\

\indent
To prove Theorem \ref{elliptic} it is enough to consider the case $q\le r$ in which
$q \le 2(d+\alpha)/d $.
 From interpolation with the conserved mass
it is clear that if $\|u\|_{L_{t}^{q_0}L_{x}^{r_0}([0,T^{*}))}=\infty$
for some admissible $(q_0,r_0)$  then  $\|u\|_{L_{t}^{q}L_{x}^{r}([0,T^{*}))}=\infty$
 for all  admissible $(q,r)$ satisfying $r_0\le r$.
 Hence if one can show \eqref{concent} with $\|u\|_{L_{t}^{q}L_{x}^{r}([0,T^{*}))}=\infty$,
 the result for $\|u\|_{L_{t}^{q_0}L_{x}^{r_0}([0,T^{*}))}=\infty$ automatically follows.

\

Let $u$ be the maximal solution
to \eqref{alphsch} over the maximal forward existence  time interval
 $[0,T^{*})$ so that \eqref{blowup-qr-1} is satisfied for an
$\alpha$-admissible pair $(q,r)$, $2<q\le r$
and $\|u\|_{L_{t}^{q}L_{x}^{r}([0,t) \times \mathbb{R}^d)}
<\infty$ for $0 < t < T^{*}<\infty$.

\noindent Then for a fixed small $\eta>0$  there is
 a strictly increasing sequence $\{t_n\}_{n=1}^{\infty}$
 in $[0,T^*)$ such that $\lim_{n\to \infty} t_n=T^\ast$
and  for every $n \in \mathbb{N}$
 \begin{equation}\label{e-3}
         \|u\|_{L_{t}^{q}L_{x}^{r}((t_n,t_{n+1}) \times \mathbb{R}^d)} = \eta.
\end{equation}
By Duhamel's formula, we have for $t \in (0,T^{*})$
\begin{equation*}
u(t,x) = e^{i(-\Delta)^{\frac{\alpha}{2}} (t-t_n)}u(t_n) \mp i \int_{t_n}^{t} e^{i(-\Delta)^{\frac{\alpha}{2}}
(t-s)}|u(s)|^{\frac{2\alpha}{d}}u (s) \, ds.
\end{equation*}
Applying Strichartz's estimate with \eqref{e-3}, we have
 \begin{eqnarray}\label{Strichartz}
        & & \|u - e^{i(-\Delta)^{\frac{\alpha}{2}} (t-t_n)}u(t_n)\|_{L_{t}^{q}L_{x}^{r}((t_n,t_{n+1}) \times \mathbb{R}^d)}\\
                  &\leq& C   \,\|u\|_{L_{t}^{q}L_{x}^{r}((t_n,t_{n+1}) \times
         \mathbb{R}^d)}^{\frac{2\alpha+d}{d}}
         = C \,  \eta^{\frac{2\alpha+d}{d}} \nonumber
\end{eqnarray}
where \eqref{r-range} holds.
Hence from  \eqref{e-3}, \eqref{Strichartz} and time
translation invariance property we obtain
\begin{equation*}
 \|e^{i(t-t_n)(-\Delta)^{\frac{\alpha}{2}}}u(t_n)\|_{L_{t}^{q}L_{x}^{r}((t_n, t_{n+1}) \times
         \mathbb{R}^d)}
  \geq  \eta -  C \,  \eta^{\frac{2\alpha+d}{d}} > \eta^{\frac{2\alpha+d}{d}}
 \end{equation*}
for sufficiently small $\eta$.

Fix $n \in \mathbb{N}$ and the time interval $(t_n,t_{n+1})$. We
denote $f = u(t_n)$ and then by the mass conservation we have \[\|f\|_{L^{2}(\mathbb{R}^d)} =
\|u_0\|_{L^{2}(\mathbb{R}^d)}.\]
Applying  Lemma \ref{square} to $f$ with $\epsilon =
  \eta^{\frac{2\alpha+d}{d}}$, there exists $\{f_\sigma\}_{1\leq \sigma \leq L}$ such that
 $\widehat{f}_\sigma$ is supported in a dyadic shell $B_{n_\sigma}$,
\begin{equation}\label{size}
|\widehat{f}_\sigma| \leq C\,\ep^{-\nu} 2^{-\frac{n_\sigma d}{2}}\end{equation}
 and
 \begin{equation}\label{p4}
 \|e^{i(t-t_n)(-\Delta)^{\frac{\alpha}{2}}}f - \sum_{s=1}^{L}e^{i(t-t_n)(-\Delta)^{\frac{\alpha}{2}}}f_\sigma
 \|_{L_{t}^{q}L_{x}^{r}(\mathbb{R} \times \mathbb{R}^d)} < \eta^{\frac{2\alpha+d}{d}},
 \end{equation}
 where $L=L(\|f\|_{L^{2}}, d, \eta)$.

 By H\"{o}lder's inequality with $\frac{2}{r}+ \frac{r-2}{r}=1$, we have
 \begin{eqnarray}\label{holder}
 &\intn\left( \intd
|u|^2|u(t,x)-\sum_{s=1}^{L}e^{i(t-t_n)(-\Delta)^{\frac{\alpha}{2}}} f_\sigma|^{r-2} dx
\right)^{\frac qr} dt\\
  &\leq \|u\|_{\tx{q}{r}((t_n,t_{n+1}) \times
         \mathbb{R}^d)}^{\frac{2q}{r}}
  \|u -
  \sum_{s=1}^{L}e^{i(t-t_n)(-\Delta)^{\frac{\alpha}{2}}}
  f_{\sigma}\|_{\tx{q}{r}}^{\frac {q(r-2)}{r}}.\nonumber
  \end{eqnarray}
 By using  H\"{o}lder's inequality with $\frac{2}{r}+ \frac{r-2}{r}=1$ again, the last term of \eqref{holder} is bounded by
 \begin{align*}
& & \|u\|_{\tx{q}{r}((t_n,t_{n+1}) \times
         \mathbb{R}^d)}^{\frac{2q}{r}} \biggl(\|u - e^{i(t-t_n)(-\Delta)^{\frac{\alpha}{2}} }u(t_n)\|_{\tx{q}{r}((t_n,t_{n+1}) \times
         \mathbb{R}^d)}^{\frac {q(r-2)}{r} }  \\
 & & \qquad \quad + \ \|e^{i(t-t_n)(-\Delta)^{\frac{\alpha}{2}}} u(t_n) -
  \sum_{\sigma=1}^{L}e^{i(t-t_n)(-\Delta)^{\frac{\alpha}{2}}}
  f_\sigma\|_{\tx{q}{r}}^{\frac {q(r-2)}{r}} \biggl)
  :=  \mathcal{E} + \mathcal{F}.
\end{align*}
In order to estimate $\mathcal{E}$ and $\mathcal{F}$, we apply \eqref{e-3}, \eqref{Strichartz} and \eqref{p4}.
Since $ \frac{(2\alpha+d)r - 4\alpha}{d}  >  r$ for $r \geq q > 2$, we see that
\begin{eqnarray}\label{small}
 \mathcal{E} + \mathcal{F} &\leq&  C \, \eta^{\frac{2q}{r}}
  \eta^{\frac{(2\alpha+d)(r-2)q}{dr}}< \frac{\eta^q}{2}.
\end{eqnarray}

\noindent We may split \eqref{e-3} into two integrals such as
 \begin{eqnarray}\label{decompose}
  \eta^q  &=& \int_{t_n}^{t_{n+1}} \left (\int_{\bbr^d} |u|^r dx \right)^{\frac
  qr}  \, dt\nonumber \\
  &\le& \  \intn\left( \intd |u|^2|\sum_{\sigma=1}^{L}e^{i(t-t_n)(-\Delta)^{\frac{\alpha}{2}}} f_\sigma|^{r-2} dx \right)^{\frac qr}
    dt \\
  & & \ + \   \intn\left( \intd |u|^2|u-\sum_{\sigma=1}^{L}e^{i(t-t_n)(-\Delta)^{\frac{\alpha}{2}}} f_\sigma|^{r-2} dx \right)^{\frac qr}
    dt.  \nonumber
\end{eqnarray}
From \eqref{small} and \eqref{decompose} we obtain that
 \[\intn\left( \intd |u|^2|\sum_{\sigma=1}^{L}e^{i(t-t_n)(-\Delta)^{\frac{\alpha}{2}}} f_\sigma|^{r-2} dx \right)^{\frac qr} dt
   \geq \frac{\eta^q}{2}.\]
Since $L = L(\|u_0\|_{L^2(\mathbb{R}^d)},
\eta)$, there exists an $n_0$ and an $f_0 = f_{n_0}$ supported on
a dyadic shell $B_k$ for some $k$ such that
 \begin{equation}\label{p-5}
  \int_{t_n}^{t_{n+1}}\, \!\!\!
\big(\int_{\mathbb{R}^d}|u(t,x)|^{2} \,
\big|e^{i(t-t_n)(-\Delta)^{\frac{\alpha}{2}}}f_0(x) \big|^{r-2} \, dx
\big)^{\frac{q}{r}}dt \geq \epsilon_{0},
\end{equation}
where we denote by $\epsilon_{0} = \frac{1}{2}
\frac{\eta^{q}}{L_{0}^{(r-2)q/r}}$.
Then from \eqref{size} we have $|\widehat f_0|\le C\,\ep^{-\nu} 2^{-\frac{k d}{2}}$.

 By Lemma \ref{tube}, there is a $L_1=L_1(\|f_0\|_{L^2},\eta)$ and a set of regions
$\{\mathcal{Q}_{n}\}_{1 \leq n \leq L_{1}}$ defined by
 $$\mathcal{Q}_{n} = \{ (t,x) \in \bbr\times \bbr^d \, ; \, t \in I_n  \mbox{ and }(x-4\pi t\xi_0)\in C_n \}, $$
where $C_{n}$ is a cube of side length $l(C_{n}) = 2^{-k}$ and $I_n$
is an interval of length $|I_{n}| = 2^{-k\alpha}$ such that
             $$\|e^{i(t-t_n)(-\Delta)^{\frac{\alpha}{2}}}f_0 \|_{L_{t}^{q}
             L_{x}^{r}({\mathbb{R}}\times {\mathbb{R}^d}\setminus \bigcup_{n=1}^{L_1} \mathcal{Q}_{n}) }
              < (\frac{\epsilon_0}{2\eta^{2q/r}})^{\frac{r}{q(r-2)}}. $$
   Then by H\"older's inequality with $\frac{2}{r} + \frac{r-2}{r} =1$ repeatedly, we have
\begin{eqnarray*}
   & &  \left\| |u|^2|e^{i(t-t_n)(-\Delta)^{\frac{\alpha}{2}}} f_0|^{r-2} \right\|_{\tx{q}{r}
   ((t_n,t_{n+1}) \times\bbr^{d} \setminus  \bigcup_{n=1}^{L_1} \mathcal{Q}_{n})}^q \\
 &\leq& \|u\|_{L_{t}^{q}L_{x}^{r}((t_n,t_{n+1}) \times \mathbb{R}^d)}^{\frac{2q}{r}} \,
 \|e^{i(t-t_n)(-\Delta)^{\frac{\alpha}{2}}}f_0 \|_{L_{t}^{q}L_{x}^{r}(\mathbb{R} \times
 \mathbb{R}^d)\setminus \bigcup_{n=1}^{L_1} \mathcal{Q}_{n})}^{\frac{q(r-2)}{r}} \\
 &<& \eta^{2q/r}\frac{\epsilon_0}{2\eta^{2q/r}} =
 \frac{\epsilon_0}{2}.
\end{eqnarray*}
Thus from  \eqref{p-5} it follows that
  \begin{equation*}
  \left\| |u|^2|e^{i(t-t_n)(-\Delta)^{\frac{\alpha}{2}}} f_0|^{r-2} \right\|_{\tx{q}{r}
   ((t_n,t_{n+1}) \times \bbr^{d} \, \cap \, (\bigcup_{n=1}^{L_1} \mathcal{Q}_{n}))}^q  \, \geq \, \frac{\epsilon_0}{2}.
   \end{equation*}
This implies that there is a region $\mathcal{Q}_0 \in \{\mathcal{Q}_{n}\}_{n=1}^{L_1}$ such
that
 \begin{equation}\label{p-60}
   \int_{(t_n,t_{n+1}) \, \cap \, I_{0}}\,  \mathcal{H}(t) \, dt  \ \geq \ \frac{1}{2L_1}\, \ep_0:= \ep_1
   \end{equation}
   where we set
   \begin{align*}
   \mathcal{H}(t) &= \big(\int_{\mathcal{Q}_0^t}|u(t,x)|^{2} \,
    |e^{i(t-t_n)(-\Delta)^{\frac{\alpha}{2}}}f(x)|^{r-2} \, dx
   \big)^{\frac{q}{r}}  \ \text{and}\\
   \mathcal{Q}^t_{0}=&\{x:(x,t)\in \mathcal{Q}_0\}.
   \end{align*}

\noindent Since $|\widehat f_0|\le C\, 2^{-\frac{k d}{2}}$ and $\widehat{f_0}$ is
supported in a dyadic shell of measure $2^{k d}$, we have
 \begin{align*}
  |e^{i(t-t_n)(-\Delta)^{\frac{\alpha}{2}}}f_0(x)|^{\frac{q(r-2)}{r}} \,& \leq \, \left( \int_{B_0} |\widehat f_0(\xi)| \, d\xi \right)^{\frac{q(r-2)}{r}}\\
  & \le C 2^{k\al}\ep^{-\frac{2\al\nu}{d}} =\, C \,
|I_0|^{-1}\ep^{-\frac{2\al\nu}{d}},
\end{align*}
where we use  $dq(r-2)/r = 2\alpha$ and $|I_0|=2^{-k\al}$.  Thus  we have
\begin{align}\label{p-6}
\mathcal{H}(t)\le C |I_0|^{-1} \ep^{-\frac{2\al\nu}{d}}\|u_0\|_{L^2_x}^{\frac{2q}{r}},
\end{align}
and  in view of \eqref{p-60}
\begin{eqnarray*}
   \epsilon_1
   &\leq& |I_0|^{-1} \ep^{-\frac{2\al\nu}{d}}\|u_0\|_{L^2_x}^{\frac{2q}{r}} (t_{n+1} - t_{n}).
 \end{eqnarray*}
 Thus we  find the lower bound
\[t_{n+1} - t_{n} \geq C  \, |I_0|\ep^{{2\al \nu}/{d}}{\epsilon_1}:= C  \, |I_0|\epsilon_2.\]
We divide the integral in the left hand side of \eqref{p-60} into two integrals such that
\begin{eqnarray*}
 \big(\int_{t_n}^{t_{n+1}-A\, |I_0|\epsilon_2}
  + \int_{t_{n+1}-A\, |I_0|\epsilon_2}^{t_{n+1}} \big)\,
    \mathcal{H}(t) \, dt .
  \end{eqnarray*}
By \eqref{p-6}, similarly  we can choose $A$ small enough so that
 \begin{eqnarray*}
  \int_{t_{n+1}-A\, |I_0|\epsilon_2}^{t_{n+1}} \mathcal{H}(t) \, dt \
  \le \ \ A \, \ep_1 \, \|u_0\|_{L^2}^{\frac{2q}{r}} \ \le \ \frac{\ep_1}{2}.
  \end{eqnarray*}
In view of this and \eqref{p-60}, we obtain that
\begin{equation*}
   \int_{ (t_n,t_{n+1}-A\, |I_0|\epsilon_2) \, \cap \,  I_{0}}\, \mathcal{H}(t) \, dt   \ \geq \ \frac{\epsilon_1}{2}.
   \end{equation*}
The inequality \eqref{p-6} leads to us that
 \begin{eqnarray*}
\frac{\epsilon_1}{2} &\leq& C \, |I_{0}| \sup_{t \, \in
(t_n,t_{n+1}-A\, |I_0|\epsilon_2)} \mathcal{H}(t)\\
 &\leq& C \,   \,{\ep_1}{\ep_2}^{-1}
\left( \sup_{t \, \in (t_n,t_{n+1}-
A\, |I_0|\epsilon_2)}\int_{\mathcal{Q}_0^t}|u|^{2} \, dx
   \right)^{\frac{q}{r}}.
 \end{eqnarray*}
Hence we obtain that
\[\sup_{t \, \in
(t_n,t_{n+1}-A\, |I_0|\epsilon_2)}\int_{\mathcal{Q}_0^t}|u|^{2} \,
 dx \geq C \, \left(\frac{\ep_2}{2}\right)^{\frac rq}.\]
 Thus, for each $t_n$ there are $t_0\in (t_n, t_{n+1}-A\, |I_0|\epsilon_2]$ and a cube $\mathcal{Q}_0^{t_{0}}$
such that
\[ \int_{\mathcal{Q}_0^{t_{0}}}|u(t_{0},x)|^{2} \, dx \geq
  \frac{C}{4} \left(\frac{\epsilon_2}{2}\right)^{\frac{r}{q}}.\]
Since $l(\mathcal{Q}_0^{t_{0}}) = |I_0|^{\frac{1}{\al}}$, then $\mathcal{Q}_0^{t_{0}}$ is
contained in a ball of radius $C_d \, |I_0|^{\frac{1}{\al}}$. Since $t_{n+1}-t_0 \ge C\, \ep_2|I_0|$,
 \[  \ep_2^{\frac{1}{\al}}|I_0|^{\frac {1}{\al}} \le C \, (t_{n+1}-t_0)^{\frac 1\alpha} \le C \, (T_* -t_0)^{\frac 1\alpha}.\]
 Hence $\mathcal{Q}_0^{t_{0}}$ can be covered by a
finite number (depending  on $\eta, d$ and $\|u_0\|_2$) of
balls of radius $r = (T^{*} -t_{0})^{\frac{1}{\alpha}}$.

Therefore, there exists
 $x_0  \in \mathbb{R}^{d}$ such that
$$ \int_{B(x_0, (T^{*} -t_{0})^{\frac{1}{\alpha}})}|u(t_{0},x)|^{2} \, dx \geq
\varepsilon,$$
where $\varepsilon$ is  $\varepsilon \, (\|u_0\|_{L^2(\mathbb{R}^d)}, d, \eta)$ and independent of $t_n$.
This completes the proof. \qed

\

\textit{Proof of Theorem \ref{hartree}}.\\  We proceed as in proof of Theorem \ref{elliptic}.
Let $u$ be the maximal solution
to \eqref{hartee} over the maximal forward existence  time interval
 $[0,T^{*})$ so that \eqref{blowup-qr-1} holds for some Strichartz
admissible pairs $(q,r)$ satisfying \eqref{hart-qr},
and $\|u\|_{L_{t}^{q}L_{x}^{r}([0,t) \times \mathbb{R}^d)}
<\infty$ for $0 < t < T^{*}<\infty$.
Let $\eta$ and sequence $t_1,\dots,t_n,\dots$ be given as before  such that
$ t_n\nearrow T^\ast$
and \eqref{e-3} is satisfied  for every $n \in \mathbb{N}$.
By the  Duhamel's
formula we may write for $t \in (0,T^{*})$
\begin{equation*}
u(t) = e^{i\alphla (t-t_n)}u(t_n)  \pm i \int_{t_n}^{t}
e^{i(t-s)\alphla}[(|x|^{-2}*|u(s)|^2)u(s) ]\, ds.
\end{equation*}
We need to show the similar estimate as \eqref{Strichartz} for the solution of Hartree equation. That is to say,
\textit{for the  solution $u$ of \eqref{hartee} there is a constant
 $C>0$ such that
\begin{equation}\label{claim}
\big\| \int_{t_n}^t e^{i(t-s)\alphla}[(|x|^{-\al}*|u(s)|^2)u(s) ]
ds\big\|_{ L^q_tL^r_x( [t_n,t_{n+1}]\times \mathbb{R}^d)} \,\le C
\,  \eta^{1+\theta}
\end{equation}
for $(q,r)$ satisfying \eqref{hart-qr} and for some $0<\theta<1$.}
We note that the inequality above is obtained by repeating the local wellposement argument.
See the argument around \eqref{ok}\footnote[2]{In fact,  with $v=0$ one can easily show \eqref{claim} with $\theta=0$.} in Section \ref{intro}.
After achieving this we only need to deal with the homogeneous part of the solution to show the mass concentration.
Hence, the remaining parts  are the same as those for Theorem \ref{elliptic}.
We omit the details.

 \section{Appendix}
 To prove Lemma \ref{square} and Lemma \ref{tube} we modify Bourgain's arguments in \cite{bo2} (also see \cite{bv})
 for  the Schr\"odinger operator  of higher orders $\alpha$ with $\alpha > 2$. The proof of Lemma \ref{square} relies on
 Proposition \ref{lem-con} which is obtained in Section 2. For the Proof of Lemma \ref{tube} the required strengthened estimate  is given by
\eqref{p0}  because  the $\alpha$ admissible pairs are contained in the range $2/q+d/r\le d/2$.

\begin{proof}[Proof of Lemma \ref{square}]
From \eqref{lem-assumption} and Proposition \ref{lem-con}, we see that there are $0<\theta<1$ and $p < 2$ such that
\begin{eqnarray*}
\epsilon\,\le\,\big\|
e^{it(-\triangle)^{\frac{\alpha}{2}}}f\big\|_{{L^{q}_tL^{r}_x}}\,\le\,\|f\|_{L^2}^{1-
\theta}\big(\sup_{k}\,
2^{kd(\frac{1}{2}-\frac{1}{p})}\|\widehat f {\chi_{B_{k}}}\|_{L^p}\big)^{\theta}.
\end{eqnarray*}
So there exists a dyadic shell $B_{n_1}$  for some $n_1$ such that
\begin{eqnarray}\label{10.5}
\| \, \widehat{f} \, \|_{L^p(B_{n_1})}^p\ge
\big(\epsilon^{\frac{1}{\theta}} \, 2^{ n_1d(\frac{1}{p}-\frac{1}{2})} \, \|f\|_{L^2}^{1-\frac{1}{\theta}}\big)^p.
\end{eqnarray}
Applying Lemma \ref{nonconcent} to $\widehat f$ and $ B_{n_1}$, we have
\[\ep^{\frac1\theta}\|f\|_2^{1-\frac 1\theta}
\lesssim \|\widehat f_{B_{n_1}}^\lambda\|_2 \]
when $\lambda\sim |B_{n_1}|^{-\frac12} \ep^{-\frac p{\theta(2-p)}}\|f\|_2^{\frac p{\theta(2-p)}+1}$.
We now define $f_1$ by $\widehat f_1 = \widehat f_ {B_{n_1}}^{\la}$ and insert $|B_{n_1}|\sim 2^{n_1 d}$. If $\|\propa (f-f_1)\|_{\tx{q}{r}} \le \ep$, we are done by setting $\nu = \frac p{\theta(2-p)}$,
$\mu= \frac p{\theta(2-p)}+1$.
 The property $(3)$ follows from disjoint supports of $\widehat{f_1}$ and $\widehat f - \widehat{f_1}$. On the other hand,
if $\|\propa (f-f_1)\|_{\tx{q}{r}} \ge \ep$, we repeat the above argument for $f-f_1$ to find $f_2$, $B_{n_2}$, $\la$
such that $
  |\widehat{f_2}|\le \la, \quad \la \sim |B_{n_2}|^{-\frac12}\ep^{-\nu}\fl^{\mu},$ and
$ \ep \|f\|_2^{1-\frac1{\theta}}\le \ep \|f-f_1\|_2^{1-\frac1{\theta}} \lesssim \|f_2\|_2,$
where the first inequality follows from $\|f\|_2= \|f_1\|_2 + \|f-f_1\|_2$.  The $L^2$ orthogonality holds as well,
$\|f-f_1\|_2^2 = \|f_2\|_2^2+\|f-(f_1+f_2)\|_2^2.$
\\ \indent Recursively we can find $f_k$ supported on $B_{n_k}$ in the frequency
space for  $k=1,2,\dots,N$ such that
\begin{align*}
& |\widehat{f_k}| < C 2^{-\frac{dn_k}{2}}\ep^{-\nu}\|f\|_2^{\mu}, \quad \|f_k\|_2 \ge \ep\|f\|_2^{1-\frac 1\theta},\\
& \|f\|_2^2 = \sn \|f_k\|_2^2 +\|f - \sn f_k\|_2^2.
\end{align*}
This process will stop within a finite number of steps. The number of steps depends on $\epsilon$ and $\|f\|_{L^2}$ because
\begin{align*}
\big\|e^{it(-\triangle)^{\frac{\alpha}{2}}}f-\sum_{j=1}^ne^{it(-\triangle)^{\frac{\alpha}{2}}}f_j\big\|_{{L^{q}_tL^{r}_x}}^2
&\le
C  \|f-\sum_{j=1}^nf_j\|_{L^2}^2\\
&= C  (\|f\|_{L^2}^2-\sum_{j=1}^n\|f_j\|_{L^2}^2)\\
&\le C  \big(\|f\|_{L^2}^2- n \, C \|f\|_{L^2}^{-a}\epsilon^b\big).
\end{align*}
This completes the proof.
\end{proof}

\begin{proof}[Proof of Lemma \ref{tube}]
  We follow closely the argument for the proof Lemma $3.3$ in \cite{bv}. Let $g' \in L^2(\mathbb R^d)$ be
   the normalized function of $g$ defined by $\widehat {g'}(\xi')= 2^{\frac {kd}{2}} \widehat g(2^k \xi').$
   Then $\mbox{supp } \widehat{g'} \subset B_1$, $\|g'\|_{L^2}= \|g\|_{L^2}$ and $|\widehat{g'}|<1$.
   We see that
   \begin{align*}
     e^{i2^{k\al}t (-\Delta)^{\frac {\al}{2}}}g'(2^k x) &=
     2^{\frac{kd}{2}}\int e^{i2^kx\cdot \xi + it2^{k\al}|\xi|^{\al}} \widehat g(2^k\xi) d\xi\\
     &= 2^{-\frac{kd}{2}}\int e^{ix\cdot \xi + it|\xi|^{\al}} \widehat g(\xi) d\xi
     = 2^{-\frac{kd}{2}} \propa g(x).
   \end{align*}
   That is to say,
   \begin{align}\label{A}
  \propa g(x) = 2^{\frac {kd}{2}} \propah g(x')
  \end{align}
   by the change of variable  $(t,x)\to (t',x')= (2^{k\al}t, 2^k x)$.  \\

  We will keep track of the free evolution of $g'$.
 Let $E \subset \bbr\times \mathbb R^d$ be the set $\{(t',x'):|\propah g'(x')|<\la\}$ for a given $\la$.
 We have
  \begin{align*}
\| \propah g'\|_{\txp{q}{r}(E)}^q
  & = \int_{\bbr}\big( \int_{\{x' : |\propah g'(x')| \, < \, \la\}}
  |\propah g'(x')|^{r^* + r-r^*} dx' \big)^\frac qr dt'\\
  &\le \la^ {(r-r^*)\frac qr} \int \left( |\propah g'(x')|^{r^* } dx' \right)^\frac qr dt'.
  \end{align*}
Since $\alpha>2$, we now note that
the  $\al$-admissible line is properly contained in the region of $2/q+d/r\le d/2$. Hence,
we can pick up a pair $(q^*, r^*)$ in the region $\frac {d}{r^*} + \frac{2}{q*} \le \frac{d}{2}$
 such that $q^*<q$, $r^*<r$  and  ${r^*}/{q^*}= r/q $. Such choice may not be possible for the end point $\frac 1r= \frac{d-\al}{2d}$ but it was excluded because we are assuming $q>2$ and $r\neq \infty$ (see \eqref{r-range}).
   Then for $\al$-admissible $(q,r)$, \eqref{p0} yields
 \[\| \propah g'\|_{\txp{q}{r}(E)}^q \ \le \ C\, \la^ {(r-r^*)\frac qr}\|\widehat{g'}\|_{L^2} \ \le
 C\la^ {(r-r^*)\frac qr}\|\widehat{g'}\|_{L^{\infty}},\]
 where the second inequality follows from the fact that $ \mbox{supp } \widehat{g'} \subset B_1$. Since $r^*<r$, by choosing $\la=\la(C_0,\ep)$ small enough, we have
 \[\| \propah g'\|_{\txp{q}{r}(\tilde E)}^q \le \ep^ q\]
 where $\widetilde E = \{(t',x'): |e^{it'\Delta} g'(x')| < 2\la\}$.\\

 Due to the normalization,
 $\mbox{supp} \, \widehat {g'} \subset B_1$
and $\| \widehat {g'} \|_{L^\infty} \le C_0$. Hence the  function $(x,t)\to e^{it\alphla} g'(x)$ is smooth with bounded derivatives. In particular,  the map
 \[ e^{it'\alphla} g'(x)= \int_{\bbr^d} e^{2\pi i (x\cdot \xi -2\pi t|\xi|^{\al})} \, \widehat{g'}(\xi) \, d\xi\]
 is Lipschitz. That is,
\[ | e^{i t' \alphla} g'(x') - e^{i t'' \alphla} g(x'') |
\ \le \ C \, ( | t' - t'' | + |x' - x''| ), \]
where $C = C(C_0 , d) \ge 1$.
Hence, if $(t', x') \in  E$ and  $|x' - x''|,|t' - t''| \le \frac {\la}{2C} < \frac 1 2$, then $(t'',x'')$ is in $\widetilde E$. In other words,
for $(t', x')\in (\bbr\times \bbr^d)\backslash \tilde E$, there is a space-time cube $P= J \times K $
centered at $(t',x')$ with  $|J| = \frac {\la}{C}$ and $\it{l} (K) = \frac {\la}{C}$ such that $P \in (\bbr\times \bbr^d)\backslash  E$. Let us cover $(\bbr\times \bbr^d)\backslash \tilde E$ with the family of $(P_r)_{r\in I}$ such that $\mbox{Int} (P_r) \cap \mbox{Int} (P_s) = \emptyset$ for $r\neq s$, and
\begin{align}\label{pr}
  \{ |\propah g'(x') | \ge 2 \la \} \subset \bigcup_{r \in I} P_r
  \subset \{ |\propah g'(x') | \ge \la \}.
\end{align}

 where $\mbox{Int} (P_r)$ denotes the interior of the set $P_r$. Note that the index set $I$ is finite.
We set $N_1 = \sharp I$. It follows from \eqref{pr} and the Strichartz's estimate  that
\begin{align*}
& N_1 \big( \frac {\la} C \big)^{d+1} =
\big| \bigcup_{r \in I} P_r \big| \le
| \{ |\propah g'(x') | \ge \la \} | \\
&\le \ \la^{- \frac{2(d+\al)}{d}}
\| e^{i t' \alphla} g'(x') \|_{L_{t,x}^{\frac{2(d+\al)}{d}}} (\bbr \times \bbr^d)^{\frac{2(d+\al)}{d}}
\\
&\le  \ C \, \la^{- \frac{2(d+\al)}{d}} \| g \|_{L^2}^{\frac{2(d+\al)}{d} }
\end{align*}
from which we deduce that $N_1 \le C(\| g \|_{L^2}, d, C_0, \epsilon)$.
Actually, since our hypothesis implies that $\| g \|_{L^2} \le C_0$,
we can also write $N_1 \le C(d, C_0, \epsilon)$. For simplicity let $\{1,\dots, N_1\}$ denote the index set $I$.
For any integer $1 \le n \le N_1$, let $(t_n, x_n)$ be the center of $P_n$ and let $I_n \subset \bbr$ be the interval of center $\frac {t_n} {2^{k\alpha}}$ with
$|I_n | = \frac {1} {2^{k\alpha}}$. Also set $I_n ' =2^{k\alpha}I_n$.
Let $C_n \in \mathcal C$ of center $2^{-k} x_n$ with $\ell(C_n) = 2^{-k}$ and
let $C_n ' =2^{k}C_n$. Finally let $\mathcal{Q}_{n}$ be defined by \eqref{qn}.
Then from the choice of $\lambda$ it follows that
 \[\| \propah g'\|^q_{\txp{q}{r}(\bbr^{d+1}\backslash \bigcup _{n=1}^{N_1} I_n'\times C_n')} <\ep^q .\]
 By \eqref{A} and reversing the change of variables $(t',x')\to (t,x)$, we have
 \begin{align*}
 &\| \propa g\|^q_{\tx{q}{r}(\bbr^{d+1}\backslash \bigcup _{n=1}^{N_1} \mathcal{Q}_{n})} \\
 &= 2^{k( d/2-d/r-\al/q)}\| \propah g'\|^q_{\txp{q}{r}(\bbr^{d+1}\backslash \bigcup _{n=1}^{N_1} I_n'\times C_n')}
   <\ep^q
 \end{align*}
 since $(q,r)$ is admissible.
This concludes the proof of the lemma.
\end{proof}


\end{document}